\documentclass[leqno,12pt]{amsart}
\usepackage{latexsym}
\usepackage{mathrsfs}
\usepackage{amsmath}
\usepackage{amssymb}
\usepackage[bookmarks]{hyperref}
\usepackage{pstricks,pst-plot}

\newtheorem{theorem}{Theorem}[section]
\newtheorem{lemma}[theorem]{Lemma}
\newtheorem{corollary}[theorem]{Corollary}
\newtheorem{proposition}[theorem]{Proposition}

\theoremstyle{definition}

\newtheorem{definition}[theorem]{Definition}

\newcommand{\C}{\mathbb{C}}

\newcommand{\N}{\mathbb{N}}

\newcommand{\Z}{\mathbb{Z}}

\newcommand{\bthm}{\begin{theorem}}
\newcommand{\ethm}{\end{theorem}}
\newcommand{\blem}{\begin{lemma}}
\newcommand{\elem}{\end{lemma}}
\newcommand{\bcor}{\begin{corollary}}
\newcommand{\ecor}{\end{corollary}}
\newcommand{\bprop}{\begin{proposition}}
\newcommand{\eprop}{\end{proposition}}
\newcommand{\bdefn}{\begin{definition}}
\newcommand{\edefn}{\end{definition}}
\newcommand{\bpf}{\begin{proof}}
\newcommand{\epf}{\end{proof}}

\def\vep {\varepsilon}
\def \sm {\setminus}

\def\h#1{\widehat {#1}}

\def\tA {\tilde A}

\def\itemskip {\vskip 3pt plus 2 pt minus 1 pt}
\def \ma {\mathfrak{M}_A} 
\def \maf {\mathfrak{M}_{A_\sf}} 
\def \mb {\mathfrak{M}_B} 
\def \mad {\mathfrak{M}_{A^D}}

\def\sf{{\mathscr F}}
\def\der{(d^{(k)})_{k=0}^n}
\def\Der{(D^{(k)})_{k=0}^n}
\def\dernum #1{(d_{#1}^{(k)})_{k=0}^n}
\def\dnonum {d^{(k)}}
\def\Dnonum {D^{(k)}}
\def\dnum #1{d_{#1}^{(k)}}
\def\dtnum #1{\tilde d_{#1}^{(k)}}

\def\qstar #1{q_{#1}^*}

\def\Asf{A_\sf}
\def\xcf{X\times\C^\sf}

\def\aa{A_\alpha}
\def\xa{X_\alpha}
\def\piab{\pi_{\alpha,\beta}}
\def\fa{\sf_\alpha}

\newcommand{\ra}{\rightarrow}

\newcommand{\ol}{\overline}

\begin{document}
\title[One-point Gleason parts and point derivations]{One-point Gleason parts and\\ point derivations in uniform algebras}

\author{Swarup N. Ghosh}
\address{Department of  Mathematics,
Southwestern Oklahoma State University,
Weatherford, OK 73096, USA}
\email{swarup.ghosh@swosu.edu}

\author{Alexander J. Izzo}
\address{Department of Mathematics and Statistics, Bowling Green State University, Bowling Green, OH 43403}
\email{aizzo@bgsu.edu}
\thanks{The second author was partially supported by a Simons collaboration grant and by NSF Grant DMS-1856010.}


\begin{abstract}
It is shown that a uniform algebra can have a nonzero bounded point derivation while having no nontrivial Gleason parts.  Conversely, a uniform algebra can have a nontrivial Gleason part while having no nonzero, even possibly unbounded, point derivations.
\end{abstract}

\maketitle

%%%%%%%%%%%%%%%%%%%%%%%%%%%%%%%%%%%%%%%%
%
%
%.                              INTRODUCTION
%
%
%%%%%%%%%%%%%%%%%%%%%%%%%%%%%%%%%%%%%%%%

\section{Introduction}

Let
$X$ be a compact Hausdorff space, and let $C(X)$ be the algebra of all continuous complex-valued functions on $X$ with the supremum norm
$ \|f\| = \sup\{ |f(x)| : x \in X \}$.  A \emph{uniform algebra} $A$ on $X$ is a closed subalgebra of $C(X)$ that contains the constant functions and separates
the points of $X$.  There is a general feeling that a uniform algebra $A$ on $X$ either is $C(X)$ or else there 
is a subset of the maximal ideal space of $A$ that can be given the structure of a complex manifold on which the functions in $A$ are holomorphic.
However, it is well known that this feeling is not completely correct. 
One is therefore led to consider weaker forms of analytic structure.  Perhaps the two most common of these are nonzero point derivations and nontrivial Gleason parts.  Thus the question arises as to how these two weak forms of analytic structure are related.  More precisely, does the presence of one of these two weak forms of analytic structure imply the presence of the other?  The main purpose of this paper is to show that the answer is \emph{no}: Either form can be present in the absence of the other.

\begin{theorem}\label{main-theorem-d1}
There exists a uniform algebra $B$ on a compact metrizable space such that there exists a nonzero bounded point derivation on $B$ but $B$ has no nontrivial Gleason parts.
\end{theorem}

\begin{theorem}\label{main-theorem-p1}
There exists a uniform algebra $B$ on a compact Hausdorff space such that $B$ has a nontrivial Gleason part but there are no nonzero point derivations on $B$. \end{theorem}

Theorem~\ref{main-theorem-p1} and its proof below were found by Garth Dales and Joel Feinstein in response to a question posed to them by the second author.  We thank Dales and Feinstein for allowing us to present their result in our paper so that the complementary Theorems~\ref{main-theorem-d1} and~\ref{main-theorem-p1} appear together in a single paper.

Note that the algebra in Theorem~\ref{main-theorem-d1} has a nonzero \emph{bounded} point derivation while the algebra in Theorem~\ref{main-theorem-p1} not only has no nonzero bounded point derivations but moreover has no nonzero, possibly unbounded, point derivations.  In fact, a uniform algebra with a nontrivial Gleason part but no nonzero \emph{bounded} point derivations was constructed by John Wermer long ago \cite{W1}.  Wermer's example is $R(K)$ for a certain compact planar set $K$.  Such a uniform algebra necessarily has nonzero unbounded point derivations \cite[Corollary~3.3.11 and Theorem~3.3.3]{Browder}.

Obviously Theorems~\ref{main-theorem-d1} and~\ref{main-theorem-p1} contain the weaker statement that \emph{at a particular point} $x$ of the maximal ideal space of a uniform algebra the condition that there is a nonzero point derivation at $x$ and the condition that $x$ lie in a nontrivial Gleason part are independent of each other.  The statement that there need not be a nonzero point derivation at a point in a nontrivial Gleason part seems to be new.  That a point at which there is a nonzero bounded point derivation need not belong to a nontrivial Gleason part was shown by Stuart Sidney long ago \cite[Example~5.13]{Sidney}.  However, in contrast to the algebra in Theorem~\ref{main-theorem-d1}, Sidney's uniform algebra does have nontrivial Gleason parts, and in fact, its maximal ideal space contains many analytic discs.

Theorem~\ref{main-theorem-p1} should be contrasted with the theorem of Andrew Browder \cite{Browder2} (see also \cite[Theorem~1.6.2]{Browder}) that if a point $x$ of the maximal ideal space $\mb$ of a uniform algebra $B$ is non-isolated in the metric topology on $\mb$, then there must be a nonzero (possibly unbounded) point derivation at $x$.

We will show that the uniform algebras in Theorems~\ref{main-theorem-d1} and~\ref{main-theorem-p1} can be taken to satisfy additional conditions. Specifically we will prove the following two results that contain Theorems~\ref{main-theorem-d1} and~\ref{main-theorem-p1}.

\begin{theorem}\label{main-theorem-d2}
There exists a normal uniform algebra $B$ on a compact metrizable space $X$ and a point $x\in X$ such that $B$ has a nondegenerate bounded point derivation  of infinite order at $x$ and $B$ has bounded relative units at every point of $X\sm\{x\}$.
\end{theorem}

\begin{theorem}\label{main-theorem-p2}
For each integer $n\geq2$, there exists a strongly regular uniform algebra
$B$ on a compact Hausdorff space $X$ such that $B$ has a Gleason part $P$ that has exactly $n$ elements, $B$ has bounded relative units at every point of $X\sm P$, and there are no nonzero point derivations on $B$. 
\end{theorem}

That Theorems~\ref{main-theorem-d2} and~\ref{main-theorem-p2} strengthen Theorems~\ref{main-theorem-d1} and~\ref{main-theorem-p1} is a consequence of known results recalled in Section~\ref{prelim} below.  The reader should compare Theorem~\ref{main-theorem-d2} with \cite[Theorems~5.1 and~5.3]{F1} of Feinstein.

In contrast to the situation in Theorems~\ref{main-theorem-d1} and~\ref{main-theorem-d2}, it is unknown whether the space on which the uniform algebras in 
Theorems~\ref{main-theorem-p1} and~\ref{main-theorem-p2} are defined can be taken to be metrizable.  
It is well known that each point of a nontrivial Gleason part is a nonpeak point, and it seems to be a difficult open question whether there exist uniform algebras (with or without nontrivial Gleason parts) on a metrizable space having no nonzero, possibly unbounded, point derivations at a nonpeak point.
If in Theorem~\ref{main-theorem-p1} we relax the requirements on $B$ to allow \emph{unbounded} point derivations, then Wermer's example shows that metrizability can be achieved.  Wermer's example does not satisfy the additional conditions given in Theorem~\ref{main-theorem-p2}.  However, a modification of the proof of Theorem~\ref{main-theorem-p2} shows that subject to allowing unbounded point derivations, metrizability can be achieved there as well.

\begin{theorem}\label{main-theorem-p3}
For each integer $n\geq2$, there exists a strongly regular uniform algebra
$B$ on a compact \emph{metrizable} space $X$ such that $B$ has a Gleason part $P$ that has exactly $n$ elements, $B$ has bounded relative units at every point of $X\sm P$, and there are no nonzero \emph{bounded} point derivations on $B$. 
\end{theorem}

In the next section we define various terms already used above and present other needed background and preliminary results.  In Sections~\ref{Cole} we discuss Brian Cole's method of root extensions, which we will use in constructing our examples.  Finally, Theorems~\ref{main-theorem-d1} and~\ref{main-theorem-d2} are proved in Section~\ref{d}, while Theorems~\ref{main-theorem-p1}, \ref{main-theorem-p2}, and~\ref{main-theorem-p3} are proved in Section~\ref{p}.  Notations introduced in Sections~\ref{prelim} and~\ref{Cole} will be used in Sections~\ref{d} and~\ref{p} without further comment.

%%%%%%%%%%%%%%%%%%%%%%%%%%%%%%%%%%%%%%%%
%
%
%.                              LEMMAS
%
%
%%%%%%%%%%%%%%%%%%%%%%%%%%%%%%%%%%%%%%%%

\section{Preliminaries}\label{prelim}

In this section we introduce terminology, notation, and conventions that we will use.  We also present some results we will need.

Throughout the paper all spaces will tacitly be required to be Hausdorff.  Throughout this section $A$ will be a uniform algebra on a compact space $X$ and $x$ will be a point of $X$.

We tacitly regard $X$ as a subspace of the maximal ideal space $\ma$ of $A$ by identifying each point of $X$ with the corresponding point evaluation functional.  When convenient, we will also tacitly regard $A$ as a uniform algebra on $\ma$ via the Gelfand transform.  When clarity seems to require it, the Gelfand transform of a function $f$ in $A$ will be denoted in the customary way by $\h f$.

The point $x$ is said to be a \emph{peak point} for $A$ if there is a function $f$ in $A$ such that $f(x)=1$ and $|f(y)|<1$ for every $y\in X\sm \{x\}$.  The point $x$ is said to be a \emph{generalized peak point} if for every neighborhood $U$ of $x$ there exists a function $f$ in $A$ such that $f(x)=\|f\|=1$ and $|f(y)|<1$ for every $y\in X\sm U$.  When the space $X$ is metrizable, the notions of peak point and generalized peak point coincide.

For $\phi\in \ma$ we define the ideals $M_\phi$ and $J_\phi$ by
\begin{align} 
M_\phi &= \{\, f\in A: \h f(\phi)=0\,\}  \nonumber  \\ 
\intertext{and} 
J_\phi &= \{ \, f\in A: \h f^{-1}(0)\ \hbox{contains a neighborhood of $\phi$ in $\ma$}\}. \nonumber 
\end{align}

The uniform algebra $A$ is \emph{normal} on $X$ if for each pair of disjoint closed subsets $K_0$ and $K_1$ of $X$ there exists a function $f$ in $A$ such that $f|K_0=0$ and $f|K_1=1$.  It is well known \cite[Theorem~27.3]{S1} that if $A$ is normal on $X$ then $X=\ma$.  The uniform algebra $A$ is \emph{strongly regular} at the point $x$ if $J_x$ is dense in $M_x$.  The uniform algebra $A$ is \emph{strongly regular} if it is strongly regular at every point of $X$.  It was proven by Donald Wilken that every strongly regular uniform algebra is normal \cite[Corollary~1]{Wilken}.

The uniform algebra $A$ has \emph{bounded relative units} at the point $x$ if there exists a positive constant $C$ such that for each compact subset $K$ of $X\sm \{x\}$ there exists a function $f$ in $J_x$ such that $f|K=1$ and $\|f\|\leq C$.  
We will need the following result of Joel Feinstein \cite[Proposition~1.5]{F1}.

\blem\label{bru-implies-gen-peak-point}
If $A$ has bounded relative units at $x$, then $x$ is a generalized peak point for $A$ and $A$ is strongly regular at $x$.
\elem

We will  also need the following lemma of Feinstein and Heath \cite[Lemma~4.3]{FH}.

\blem\label{F-H}
Let $A$ be a uniform algebra on a compact space $Y$, and let $y$ be a point of $Y$.  Suppose that, for each compact subset $E$ of $Y\sm\{y\}$, there exists an open neighborhood U of $y$ and an $f\in A$ such that
\begin{enumerate}
\item[(i)] $f|U=1$.
\item[(ii)] $f|E=0$.
\item[(iii)] For each $k\in\N$ there is a $g\in A$ with $g^{2^k}=f$.
\end{enumerate}
Then $A$ has bounded relative units at $y$.
\elem

An ideal is said to be \emph{primary} if it is contained in a unique maximal ideal.  (This use of the term \emph{primary} is unrelated to its use in commutative algebra.)  If $I$ is a primary ideal contained in $M_x$, then $I$ is said to be \emph{local} if $I$ contains $J_x$.  (Observe that this condition is equivalent to the statement that whether a function $f\in A$ belongs to $I$ depends only on the germ of $f$ at $x$.)  The notion of localness can be generalized to arbitrary ideals in $A$, but we omit the general definition as we will have no need of it.  We will, however, need the following standard result \cite[Proposition~4.1.20(iv)]{D} in the special case of primary ideals.  

\blem\label{Garth's-book}
Every ideal in a normal uniform algebra is local.
\elem

The \emph{Gleason parts} for the uniform algebra $A$ are the equivalence classes in the maximal ideal space of $A$ under the equivalence relation $\phi\sim\psi$ if $\|\phi-\psi\|<2$ in the norm on the dual space $A^*$.  (That this really is an equivalence relation is well-known but {\it not\/} obvious.)
We say that a Gleason part is \emph{nontrivial} if it contains more than one point.

The following lemma is standard.  (See \cite[Lemmas~2.6.1]{Browder}.)

\begin{lemma}\label{samepart}
Two multiplicative linear functionals $\phi$ and $\psi$ on a uniform algebra $A$ lie in the same Gleason part if and only if
$$\sup\{|\psi(f)|: f\in A, \|f\|\leq 1, \phi(f)=0\}< 1.$$
\end{lemma}

For $\phi$ in $\ma$, a \emph{point derivation} on $A$ at $\phi$ is a linear functional $\psi$ on $A$ satisfying the identity 
\begin{equation}\label{derivation-eq}
\psi(fg)=\psi(f)\phi(g) + \phi(f)\psi(g)
\end{equation}
for all $f$ and $g$ in $A$.
A point derivation is said to be \emph{bounded} if it is bounded (continuous) as a linear functional.
Now let $n$ be a positive integer or $\infty$.  A \emph{point derivation of order} $n$ at $\phi$ is a sequence $d=\der$ of linear functionals on $A$ such that for all $f$ and $g$ in $A$
\begin{align}\label{0}
d^{(0)}f&=f(\phi)  \\ 
\intertext{and}
\dnonum(fg)&=\sum_{j=0}^k (d^{(j)}f) (d^{(k-j)}g)\qquad \hbox{for all $k=1,2,\ldots$}. \label{k}
\end{align}
The point derivation $d$ is \emph{bounded} if each $d^{(k)}$ is bounded.  The point derivation $d$ is \emph{nondegenerate} if $d^{(1)}\neq 0$.  We define the \emph{kernel} $\ker d$ of the point derivation $d=\der$ by
\[
\ker d =\{\, f\in A: d^{(k)}f=0\ \hbox{for all $k=0,1,2,\ldots$}\}.
\]
When $d$ is nondegenerate the functionals $d^{(0)}, d^{(1)}, d^{(2)}, \ldots$ are linearly independent.

There is some ambiguity in our use of the term \lq\lq point derivation\rq\rq\  since it can refer either to a single linear functional satisfying equation~(\ref{derivation-eq}) or a sequence of linear functionals satisfying equations~(\ref{0}) and~(\ref{k}).  The careful reader will be able to discern which meaning is intended from the context.  Clearly a linear functional $d^{(1)}$ that is a point derivation at $\phi$ can be identified with a point derivation $d=(d^{(k)})_{k=0}^1$ of order 1 at $\phi$ by taking $d^{(0)}$ to be the functional of evaluation at $\phi$.

It is standard \cite[p.~64]{Browder} that a linear functional $\psi$ on $A$ is a point derivation at $\phi$ if and only if $\psi$ annihilates $M_\phi^2$ and the constant functions, and hence there exists a nonzero point derivation at $\phi$ if and only if $M_\phi^2\neq M_\phi$, and there exists a \emph{bounded} point derivation at $\phi$ if and only if $\ol {M_\phi^2}\neq M_\phi$.  It is elementary that the kernel of a point derivation $d=(d^{(k)})_{k=0}^n$ is an ideal, and for $n$ finite, the kernel of $d$ contains $M_\phi^{n+1}$.  Consequently, the kernel of a point derivation of finite order is a \emph{primary} ideal.  However, the kernel of a point derivation of infinite order can fail to be primary.  However, the following simple consequence of Lemma~\ref{Garth's-book} insures that on a normal uniform algebra even the kernel of a point derivation of infinite order is primary.

\blem\label{localness}
Let $A$ be a normal uniform algebra, and let $d=\der$ be a derivation of order $n$, with $1\leq n\leq \infty$, on $A$.  Then $\ker d$ is a local, primary ideal.
\elem

\bpf
The ideal $\ker d$ is local by Lemma~\ref{Garth's-book}, and we have noted above that the kernel of every finite order point derivation is primary.  We must show that $\ker d$ is primary in the case $n=\infty$.  Denote the kernel of $d=(d^{(k)})_{k=0}^\infty$ by $I$, and for each finite $r=1,2,\ldots$ denote the kernel of the finite order point derivation $(d^{(k)})_{k=0}^r$ by $I_r$.  Then $I=\bigcap_{r=1}^\infty I_r$.  Let $x$ denote the point at which the derivation $d$ is located.  Then $I_r\supset J_x$ for every $r=1,2,\ldots$.  Therefore, $I\supset J_x$.  Because $A$ is normal, $J_x$ is primary.  Consequently, $I$ is primary as well.
\epf

It is well known (and obvious from Lemma~\ref{samepart}) that every generalized peak point is a one-point Gleason part.  It is also well known that at a generalized peak point there are no nonzero point derivations \cite[Section~1--6]{Browder}.

We will achieve the strong regularity of the algebras in Theorems~\ref{main-theorem-p2} and~\ref{main-theorem-p3} by using a beautifully simple method from the paper of Feinstein \cite{F1}.  Following the notation of Feinstein, we set $A_x=\ol J_x \oplus \C\cdot 1$, where $\C\cdot 1$ denotes the constant functions on $X$.  The following two lemmas are contained in \cite[Lemmas~4.1 and~4.3]{F1}.

\blem\label{4.1}
Let the uniform algebra $A$ be normal.  Then $A_x$ is a normal uniform algebra that is strongly regular at $x$.
\elem

\blem\label{4.3}
Let the uniform algebra $A$ be normal.  Let $y$ be a point of $X$ distinct from $x$.  Then $(A_x)_y=(A_y)_x=A_x\cap A_y$.
\elem

We will need the following observation whose straightforward proof is left to the reader.

\blem\label{simple}
Let the uniform algebra $A$ be normal.  Suppose $A$ has bounded relative units at a point $y\in X$.  Then $A_x$ has bounded relative units at $y$, as well.
\elem

\bcor\label{strong-reg}
Let the uniform algebra $A$ be normal, and let $x_1,\ldots, x_n$ be points of $X$.
\begin{enumerate}
\item[(i)] At each of the points $x_1,\ldots, x_n$, the uniform algebra $A_{x_1}\cap\cdots\cap A_{x_n}$ is strongly regular.
\item[(ii)] At each point of $X$ where $A$ has bounded relative units, so does $A_{x_1}\cap\cdots\cap A_{x_n}$.
\end{enumerate}
\ecor

\bpf
By an induction argument left to the reader, it follows from Lemma~\ref{4.3} that for every permutation $\sigma$ of $x_1,\ldots, x_n$ we have
\[
A_{x_1} \cap \cdots \cap A_{x_n} = \bigl ( \cdots \bigl ( \bigl(A_{\sigma(x_1)}\bigr)_{\sigma(x_2)} \bigr )_{\sigma{(x_3)}} \cdots\bigr )_{\sigma(x_n)}
\]
The corollary is then immediate from Lemmas~\ref{4.1} and~\ref{simple}.
\epf

%%%%%%%%%%%%%%%%%%%%%%%%%%%%%%%%%%%%%
%
%                                        ROOT EXTENSIONS
%
%%%%%%%%%%%%%%%%%%%%%%%%%%%%%%%%%%%%%

\section{Cole's method of root extensions}\label{Cole}

Cole's method of root extensions involves an iterative process. We begin by discussing a single step of the iteration.

Let $A$ be a uniform algebra on a compact space $X$, and let $\sf$ be a (nonempty) subset of $A$.  Endow $\C^\sf$ with the product topology.  Let $p_1:X\times \C^\sf\ra X$ and $p_f:\xcf\ra\C$ denote the projections given by $p_1(x, (z_g)_{g\in \sf})=x$ and $p_f(x, (z_g)_{g\in\sf})=z_f$.  Define $X_\sf\subset \xcf$ by
\[
X_\sf = \{\, y\in\xcf:\bigl(p_f(y)\bigr)^2=f\bigl(p_1(y)\bigr) \ \hbox{for all $f\in\sf$}\,\},
\]
and let $A_\sf$ be the uniform algebra on $X_\sf$ generated by the set of functions $\{\, f\circ p_1: f\in A\} \cup \{\, p_f: f\in\sf\}$.  On $X_\sf$ we have $p_f^2=f\circ p_1$ for every $f\in \sf$.  Set $\pi=p_1|X_\sf$, and note that $\pi$ is surjective.  There is an isometric embedding $\pi^*:A\ra A_\sf$ given by $\pi^*(f)=f\circ \pi$.

We call the uniform algebra $A_\sf$ or the pair $(A_\sf, X_\sf)$, the $\sf$-extension of $A$, and we call $\pi$ the associated surjection.  Note that if $X$ is metrizable and $\sf$ is countable, then $X_\sf$ is metrizable also.  By \cite[Theorem~1.6]{Cole}, if $X=\ma$, then $X_\sf=\maf$.  Given $x\in X$, if $\sf$ is contained in $M_x$, the $\pi^{-1}(x)$ consists of a single point.  

There is an operator $S:A_\sf\ra \pi^*(A)$ given by integrating over the fibers of $\pi$ using the measure on each fiber that is invariant under the obvious action of $(\Z/2)^\sf$ on each fiber.  See \cite{Cole} or \cite[pp.~194--195]{S1} for details.  
Rather than working with $S$, we will use the operator 
$T: A_\sf\ra A$ obtained from $S$ by identifying $\pi^*(A)$ with $A$.  
The following properties of $T$ are almost obvious.

\blem\label{key-trick}
\phantom{.}
\begin{enumerate}
\item[(i)] $\|T\|=1$.
\item[(ii)] $T\circ \pi^*$ is the identity.
\item[(iii)] Given distinct functions $f_1,\ldots, f_r\in \sf$ and a function $f\in A$,
$$T\bigl(\pi^*(f) p_{f_1}\cdots p_{f_r}\bigr)=0.$$
\end{enumerate}
\elem

The following is contained in \cite[Theorem~2.4]{F1}.

\blem\label{preserve-normal1}
If $A$ is normal, then so is $A_\sf$.
\elem

One can iterate the above extension process to obtain an infinite sequence of 
uniform algebras and then take a direct limit to obtain another uniform algebra.  That is the procedure we will use to obtain the examples in Theorems~1.1 and~1.3.  However, for some purposes that procedure is inadequate and transfinite induction is needed to obtain the desired algebra; this is the case in the proof of Theorems~1.2 and~1.4.  Then the notion of a system of root extensions is needed.

Let $\tau$ be a fixed infinite ordinal.  A \emph{system of root extensions} is a triple of indexed sets $\bigl(\{\aa\}, \{\xa\}, \{\piab\}\bigr)$ $(0\leq \alpha\leq\beta\leq\tau)$ (denoted for brevity by $\{\aa\}_{0\leq\alpha\leq\tau}$) where each $\aa$ is a uniform algebra, each $\xa$ is a compact space, and each $\piab$ is a continuous surjective map $\piab:X_\beta\ra \xa$ such that the following conditions hold:
\begin{enumerate}
\item[(i)] The equation $\pi^*_{\alpha,\beta}(f)=f\circ \piab$ defines a homomorphism of $\aa$ into $A_\beta$.
\item[(ii)] For $\alpha\leq\beta\leq\gamma$, $\piab\circ\pi_{\beta,\gamma}=\pi_{\alpha,\gamma}$, and $\pi_{\alpha,\alpha}$ is the identity on $\xa$.
\item[(iii)] For $\alpha<\tau$, there is a subset $\fa$ of $\aa$ such that $A_{\alpha+1}$ is the $\fa$-extension of $\aa$ and $\pi_{\alpha,\alpha+1}$ is the associated surjection.
\item[(iv)] For $\gamma$ a limit ordinal, $X_\gamma$ is the inverse limit of the inverse system $\{\xa, \piab\}_{\alpha\leq\beta<\gamma}$, the maps $\pi_{\alpha,\gamma}:X_\gamma\ra\xa$ are those associated with the inverse limit, and $A_\gamma$ is the closure in $C(X_\gamma)$ of $\bigcup\limits_{\alpha<\gamma} \pi^*_{\alpha,\gamma}(\aa)$.
\end{enumerate}

The existence of systems of root extensions is of course proved by transfinite induction.  A choice of the subsets $\sf_\alpha$ uniquely determines a system of root extensions.

The following is contained in \cite[Theorem~2.1]{Cole}

\blem\label{3.3}
Given a system of root extensions $\{A_\alpha\}_{0\leq \alpha\leq \tau}$ there is a linear operator $T_\tau: A_\tau\ra A_0$ such that
\begin{enumerate}
\item[(i)] $\|T_\tau\|=1$.
\item[(ii)] $T_\tau\circ \pi_{0,\tau}^*$ is the identity.
\end{enumerate}
\elem

The following is \cite[Corollary~2.9]{F1}.

\blem\label{preserve-normal2}
Given a system of root extensions $\{A_\alpha\}_{0\leq \alpha\leq \tau}$, if $A_0$ is normal, then $A_\alpha$ is normal for all $\alpha$.
\elem

%%%%%%%%%%%%%%%%%%%%%%%%%%%%%%%%%%%%%%%%
%
%
%.                              NONZERO BOUNDED POINT DERIVATION
%
%
%%%%%%%%%%%%%%%%%%%%%%%%%%%%%%%%%%%%%%%%

\section{Nonzero bounded point derivations in the absence of nontrivial Gleason parts}\label{d}

The following lemma is the key to our construction of a uniform algebra on which there is a nonzero bounded point derivation but which has no nontrivial Gleason parts.

\begin{lemma}\label{key-lemma}
Let $A$ be a uniform algebra on a compact space $X$, and let $x$ be a point of $X$.  Suppose that $d=\der$ is a nondegenerate bounded point derivation of order $n$ {\rm(}$1\leq n\leq \infty${\rm)} at $x$, and that $\sf$ is a subset of $\ker d$.  Then there is a nondegenerate bounded point derivation $D=\Der$ of order $n$ on the $\sf$-extension $\Asf$ of $A$, at the point $y=\pi^{-1}(x)$, satisfying $\Dnonum\circ \pi^*=\dnonum$ and $\|\Dnonum \| = \| \dnonum \|$ for all $k=0,1,2,\ldots$.
\end{lemma}

\bpf
Let $T:\Asf\ra A$ be as in Section~\ref{Cole}.  For each $k=0,1,2,\ldots$, define $\Dnonum:\Asf\ra\C$ by 
\[
\Dnonum=\dnonum\circ T.
\]
Clearly each $\Dnonum$ is a bounded linear functional and by Lemma~\ref{key-trick}(ii), $\Dnonum\circ\pi^*=\dnonum\circ T\circ \pi^*=\dnonum$.  To see that $\|\Dnonum\|=\|\dnonum\|$, note that $\|\Dnonum\|\leq\|\dnonum\|  \, \|T\| = \|\dnonum\|$, while also $\|\Dnonum\| \geq \|\dnonum\|$ because for each $f\in A$ we have $\Dnonum\bigl(\pi^*(f)\bigr) = \dnonum \bigl(T(\pi^*(f))\bigr)=\dnonum f$.
In particular, $D^{(1)}\neq 0$, so $D=\Der$ is nondegenerate.  Note also that $D^{(0)}f=f(y)$ for $f\in A_\sf$.  

It remains to be shown that $D$ satisfies, for each $k=1,2,\ldots$, the derivation identity
\begin{equation}\label{derivation-identity}
\Dnonum(fg)=\sum_{j=0}^k (D^{(j)}f) (D^{(k-j)}g)
\end{equation}
for all $f,g\in A_\sf$.
It suffices to prove equation~(\ref{derivation-identity}) for $f$ and $g$ belonging to
the dense subalgebra $H$ of $\Asf$ that is algebraically generated by $\pi^*(A) \cup \{\, p_f:f\in \sf\}\, $.  Functions $f$ and $g$ in $H$ can be expressed in the form 
\begin{align}
f &= \pi^*(f_0)+\sum_{u=1}^s \pi^*(f_u) F_u \nonumber  \\ 
\intertext{and}
g &= \pi^*(g_0)+\sum_{v=1}^t \pi^*(g_v) G_v \nonumber
\end{align}
where $f_0, f_1,\ldots,f_s, g_0,g_1,\ldots, g_t\in A$ and each $F_u$ and each $G_v$ is a nonempty product of distinct functions of the form $p_f$ for $f\in \sf$.

By Lemma~\ref{key-trick}, $Tf=f_0$ and $Tg=g_0$, so for each $r=0,1,2,\ldots$, 
\begin{align}
D^{(r)}f &=(d^{(r)}\circ T)(f) = d^{(r)}f_0\nonumber\\ 
\intertext{and}
D^{(r)}g &=(d^{(r)}\circ T)(g) = d^{(r)}g_0.\nonumber
\end{align}
Since for each $k=1,2,\ldots$,
\begin{displaymath}
\dnonum(f_0g_0)=\sum_{j=0}^k (d^{(j)}f_0) (d^{(k-j)}g_0),
\end{displaymath}
the proof will be complete once we show that $D^{(r)}(fg)=d^{(r)}(f_0g_0)$ for each $r$.

View $fg$ as a sum of four terms:
\begin{align}
fg =\pi^*(f_0g_0) &+ \biggl(\sum_{u=1}^s \pi^*(f_ug_0) F_u\biggr)
+ \biggl(\sum_{v=1}^t \pi^*(f_0g_v) G_v\biggr)\nonumber\\
&+ \biggl(\sum_{u=1}^s\sum_{v=1}^t \pi^*(f_ug_v)F_uG_v\biggr).\nonumber
\end{align}
By Lemma~\ref{key-trick},
\begin{equation}\label{part1}
T\bigl(\pi^*(f_0g_0)\bigr)=f_0g_0
\end{equation}
\begin{equation}\label{part2}
T\biggl(\sum_{u=1}^s \pi^*(f_ug_0) F_u\biggr)=0
\end{equation}
\begin{equation}\label{part3}
T\biggl(\sum_{v=1}^t \pi^*(f_0g_v) G_v\biggr)=0.
\end{equation}
Now for fixed $u$ and $v$, consider $T\bigl(\pi^*(f_ug_v)F_uG_v\bigr)$.
We have $F_u=p_{f_1}\cdots p_{f_a}$ and $G_v=p_{g_1}\cdots p_{g_b}$ where $f_1,\ldots,f_a$ are distinct elements of $\sf$ and $g_1,\ldots, g_b$ are also distinct elements of $\sf$.  Note that each of the sets $\{ f_1,\ldots, f_a\}$ and $\{g_1,\ldots, g_b\}$ is necessarily nonempty. If $\{ f_1,\ldots, f_a\}=\{g_1,\ldots, g_b\}$, then $F_uG_v=p_{f_1}^2\cdots p_{f_a}^2=\pi^*(f_1\cdots f_a)$, and hence
\[
(d^{(r)}\circ T)\bigl(\pi^*(f_ug_v)F_uG_v\bigr) = (d^{(r)}\circ T)\bigl(\pi^*(f_ug_v f_1\cdots f_a)\bigr) =  d^{(r)}(f_ug_v f_1\cdots f_a);
\]
the last quantity above is zero because $f_1,\ldots, f_a$ belong to the ideal $\ker d$.  If instead $\{f_1,\ldots, f_a\}\neq \{g_1,\ldots, g_b\}$, then $F_uG_v$ can be expressed as the product of a possibly empty set of elements of $\pi^*(A)$ and a \emph{nonempty} set of functions $p_{h_1},\ldots, p_{h_c}$ with $h_1,\ldots, h_c\in \{f_1,\ldots, f_a, g_1,\ldots, g_b\}$; consequently,
$T\bigl(\pi^*(f_ug_v)F_uG_v\bigr)=0$ by Lemma~\ref{key-trick}(iii).  We conclude that 
\begin{equation}\label{part4}
(d^{(r)}\circ T)\biggl(\sum_{u=1}^s\sum_{v=1}^t \pi^*(f_ug_v)F_uG_v\biggr)=0.
\end{equation}
Collectively, equations (\ref{part1})--(\ref{part4}) yield that
\[
D^{(r)}(fg) =  (d^{(r)}\circ T)(fg) =d^{(r)}(f_0g_0),
\]
as desired.
\epf

\begin{theorem}\label{main-construction}
Let $A$ be a uniform algebra on a compact Hausdorff space $Y$, and let $x_0$ be a point of $\ma$. Suppose that there is a nondegnerate bounded point derivation $d=\der$ of order $n$ with $1\leq n \leq \infty$ on $A$ at $x_0$. Then there exists a uniform algebra $A^D$ on a compact Hausdorff space $Y^D$ and a continuous surjective map $\pi:\mad\rightarrow \ma$ such that 
\itemskip
\begin{enumerate}
\item[(i)] $\pi(Y^D)=Y$.
\itemskip
\item[(ii)] The formula $\pi^*(f)=f\circ \pi$ defines an isometric embedding of $A$ into $A^D$.
\itemskip
\item[(iii)] $\pi^{-1}(x_0)$ consists of a single point which we denote by $x_\omega$.
\itemskip
\item[(iv)]There is a nondegenerate bounded point derivation $D=\Der$ on $A^D$ at $x_\omega$ that satisfies the equation $D^{(k)}\circ \pi^*=d^{(k)}$.
\itemskip
\item[(v)] There is a dense subset $\sf$ of $\ker D$ such that every member of $\sf$ has a square root in $\sf$.  If the algebra $A$ is normal, then $\sf$ can be chosen so that for every point $y\in Y^D\sm\{x_\omega\}$ and every compact subset $E$ of $Y^D\sm \{y\}$, there exists a neighborhood $U$ of $y$ and a function $f\in\sf$ such that $f|U=1$ and $f|E=0$.
\end{enumerate}
If the maximal ideal space of $A$ is metrizable, then $A^D$ can be chosen so that its maximal ideal space is metrizable as well.  If $A$ is normal, then $A^D$ can be chosen so as to be normal as well.
\ethm

\bpf
Let $\Sigma_0=\ma$, and let $A_0$ denote $A$ regarded as a uniform algebra on $\Sigma_0$.  Set $d_0=d$.  Let $\sf_0$ be a dense subset of $\ker d$.  If $\ma$ is metrizable, choose $\sf_0$ to be countable.  (If $\ma$ is nonmetrizable, one can take $\sf_0=\ker d$.)  If $A$ is normal, then applying Lemma~\ref{localness} shows that we can, and therefore we shall, choose $\sf_0$ such that for every point $y\in \Sigma_0\sm\{x_0\}$ and every compact subset $E$ of $\Sigma_0\sm\{y\}$ there exists a neighborhood $U$ of $y$ and a function $f\in \sf_0$ such that $f|U=1$ and $f|E=0$.  Now form the $\sf_0$-extension of $A_0$.  Denote the resulting uniform algebra by $A_1$, the space on which $A_1$ is defined by $\Sigma_1$, and the canonical map $\Sigma_1\ra\Sigma_0$ by $\pi_1$.  Recall that then $\pi_1^{-1}(x_0)$ consists of a single point; denote that point by $x_1$.  If $\Sigma_0$ is metrizable, then $\Sigma_1$ is metrizable.  By Lemma~\ref{key-lemma}, there is a nondegenerate bounded point derivation $d_1=\dernum 1$ of order $n$ on $A_1$ at $x_1$ such that $\dnum 1\circ\pi_1^*=\dnum 0$ and $\|\dnum 1\|=\|\dnum 0\|$ for all $k=0,1,2,\ldots$.  By Lemma~\ref{preserve-normal1}, if $A_0$ is normal, then so is $A_1$.

We then iterate the process of taking root extensions to obtain a sequence 
$\{(A_m,\Sigma_m, \pi_m, x_m, d_m, \sf_m)\}_{m=0}^\infty$, where each $A_m$ is a uniform algebra on $\Sigma_m$, if $\ma$ is metrizable so is each $\Sigma_m$, each $\pi_m:\Sigma_m\rightarrow \Sigma_{m-1}$ is a surjective continuous map, $x_m=\pi_{m}^{-1}(x_{m-1})$, $d_m=\dernum m$ is a nondegenerate bounded point derivation of order $n$ on $A_m$ at $x_m$ such that $\dnum m \circ \pi_m^* = \dnum {m-1}$ and $\| \dnum m \| = \| \dnum {m-1} \|$ for all $k=0,1,2,\ldots$, and each $\sf_m$ is a dense subset of $\ker d_m$ such that for every $f\in \sf_m$ the function $f\circ \pi_{m+1}$ is the square of a function in $\sf_{m+1}$, and if $\ma$ is metrizable, then $\sf_m$ is countable; furthermore, if $A$ is normal, then $A_m$ is normal and $\sf_{m}$ is such that for every point $y\in \Sigma_m\sm\{x_m\}$ and every compact subset $E$ of $\Sigma_m\sm\{y\}$ there exists a neighborhood $U$ of $y$ and a function $f\in \sf_m$ such that $f|U=1$ and $f|E=0$. Finally we take the inverse limit of the system of uniform algebras $\{A_m\}$.  Explicitly, we set 
$$\Sigma_\omega=\Bigl\{(y_j)_{j=0}^\infty\in \textstyle\prod\limits_{j=0}^\infty \Sigma_j: \pi_{m+1}(y_{m+1})=y_m  \hbox{\ for all } m=0, 1, 2, \ldots\Bigr\},$$
and letting $q_m:\Sigma_\omega\rightarrow \Sigma_m$ be the restriction of the canonical projection $\prod_{j=0}^\infty \Sigma_j \rightarrow \Sigma_m$, we let $A_\omega$ be the closure of $\bigcup_{m=0}^\infty \{h\circ q_m: h\in A_m\}$ in $C(\Sigma_\omega)$.  Set $x_\omega=(x_m)_{m=0}^\infty$.  Set $\pi=q_0$.  Then $\pi^{-1}(x_0)=x_\omega$.

Note that for each $m=0,1,2,\ldots$, the formula $\qstar m(f)=f\circ q_m$ defines an isometric embedding of $A_m$ into $A_\omega$, and $A_\omega$ is the closure of $\bigcup_{m=0}^\infty \qstar m(A_m)$.  Observe that $\qstar {m+1} \circ \pi_{m+1}^*=\qstar m$, and hence $\qstar m(A_m)\subset \qstar {m+1}(A_{m+1})$.
Define $\dtnum m$ ($k=0,1,2,\ldots$) on $\qstar m(A_m)$ by
\[
\dtnum m\bigl(\qstar m(h)\bigr)=\dnum m(h) \quad \hbox{for all $h\in A_m$}.
\]
Then one easily checks that $\dtnum {m+1}$ agrees with $\dtnum {m}$ on $\qstar {m}(A_{m})$ for each $m$ and $k$.  Thus, for each fixed $k$, the union of the functionals $\dtnum m$, $m=0,1,2,\ldots$, yields a well-defined linear functional $\tilde d^{(k)}$ on $\bigcup_{m=0}^\infty \qstar m(A_m)$.  Because the functionals $\dtnum m$, $m=0,1,2,\ldots$, all have the same norm $\|\dtnum 0\|$, the functional $\tilde d^{(k)}$ is bounded, and hence extends to a bounded linear functional $D^{(k)}$ on $A_\omega$.  Moreover, $D=(D^{(k)})_{k=0}^n$ is a nondegenerate bounded point derivation of order $n$ on $A_\omega$ at $x_\omega$.

By \cite[Theorem~2.3]{Cole}, the maximal ideal space of $A_\omega$ is $\Sigma_\omega$, and the inverse image under $\pi$ of the Shilov boundary for $A_0$ is the Shilov boundary for $A_\omega$.  Consequently, setting $Y^D=\pi^{-1}(Y)$ and  $A^D$ equal to the restriction algebra $A_\omega|Y^D$, we have that $A^D$ is a uniform algebra isometrically isomorphic to $A_\omega$. 
Obviously we can regard the derivation $D$ as a derivation on $A^D$.
Set $\sf=\bigcup_{m=0}^\infty \qstar m(\sf_m)$.  As the reader can verify, conditions (i)--(v) all hold.

Note that if $\Sigma_0$ is metrizable, then so is $\Sigma_\omega$.
If $A$ is normal, then so is $A^D$ by Lemma~\ref{preserve-normal2}.
\epf

\begin{lemma}\label{lemma}
Let $A$ be a uniform algebra, let $J$ be a primary ideal in $A$, and let $E=\{\, f^2: f\in J\, \}$.  If $E$ is dense in $J$, then each point of $\ma$ is a one-point Gleason part.
\end{lemma}

\begin{proof}
The proof is similar to the proof of \cite[Lemma~1.1~(i)]{Cole}.  
Let $x$ be the point of $\ma$ such that $J$ is contained in $M_x$.
Let $y\in \ma\sm\{x\}$ be arbitrary, and let $z$ be an arbitrary element of $\ma$ distinct from $y$.  By hypothesis there is a function $f$ in $J$ such that $f(y)\neq 0$.  By multiplying $f$ by a function in $A$ that vanishes at $z$ but not at $y$ and rescaling, we may assume in addition that $f(z)=0$ and $\|f\|<1$.  Since $E$ is dense in $J$, for each $n\in\N$ and $\vep>0$, there exist functions $f_1,\ldots, f_n$ in $J$ such that
$\|f-f_1^2\|<\varepsilon, \ldots, \|f_{n-1}-f_n^2\|<\varepsilon$.  Choosing $\varepsilon>0$ small enough, $f_n^{2^n}$ can be made arbitrarily close to $f$.   
Consideration of the function $f_n-f_n(z)$ then shows that there is a function $g$ in $A$ such that $g(z)=0$ and $g^{2^n}$ is arbitrarily close to $f$.  Since $|f(y)|^{2^{-n}}\ra 1$, choosing $n$ large enough $|g(y)|$ can be made arbitrarily close to $1$.  
Thus $y$ and $z$ lie in different Gleason parts by Lemma~\ref{samepart}
\end{proof}

\begin{proof}[Proof of Theorems~\ref{main-theorem-d1} and~\ref{main-theorem-d2}]
To prove Theorem~\ref{main-theorem-d1} we start with any uniform algebra $A$ that is defined on a compact metrizable space and has a nonzero bounded point derivation of order $1$.  For instance, take $A$ to be the disc algebra.  Then taking $B$ to be the uniform algebra $A^D$ given by Theorem~\ref{main-construction} yields the result.  That $B$ has no nontrivial Gleason parts is a consequence of condition (v) by Lemma~\ref{lemma}.

To prove the stronger Theorem~\ref{main-theorem-d2} we impose on our starting uniform algebra $A$ the additional requirements that $A$ be normal and that there exist a nondegenerate bounded point derivation of infinite order on $A$.  An example of a uniform algebra satisfying these requirements was given by Anthony O'Farrell \cite{OF}.  Now take $B$ to be the uniform algebra $A^D$ given by Theorem~\ref{main-construction} choosing $\sf$ as discussed in condition~(v).  The assertion in Theorem~\ref{main-theorem-d2} about bounded relative units then holds by Lemma~\ref{F-H}.
\end{proof}

%%%%%%%%%%%%%%%%%%%%%%%%%%%%%%%%%%%%%%%%
%
%
%                              NONTRIVIAL GLEASON PART
%
%
%%%%%%%%%%%%%%%%%%%%%%%%%%%%%%%%%%%%%%%%

\section{Nontrivial Gleason parts in the absence of\\ nonzero point derivations}
\label{p}

In this section we prove Theorems~\ref{main-theorem-p1}, \ref{main-theorem-p2}, and~\ref{main-theorem-p3}.  The uniform algebra we give with the properties in Theorem~\ref{main-theorem-p3} is essentially the uniform algebra constructed by Feinstein in \cite{F2}.  The uniform algebra constructed in the proofs of Theorem~\ref{main-theorem-p1} and~\ref{main-theorem-p2} is a modification of that uniform algebra.

\bpf[Proof of Theorems~\ref{main-theorem-p1}~and~\ref{main-theorem-p2}]
Start with any normal uniform algebra $A$ that has a Gleason part with at least $n$ points.  For instance take $A$ to be the normal uniform algebra constructed by Robert McKissick \cite{McK} (which is also presented in \cite[Section~27]{S1}).
Set $(A_0, X_0)=(A, \ma)$ and let $x_0^{(1)},\ldots, x_0^{(n)}$ be $n$ distinct points belonging to a common Gleason part for $A_0$.  Let $\Omega$ denote the first uncountable ordinal.

Recall that when forming the $\sf$-extension of a uniform algebra $A$ on a space $X$, if $\sf$ is contained in $M_x$ for a point $x\in X$, then $\pi^{-1}(x)$ consists of a single point.  Consequently, one easily sees that there is a system of root extensions $\bigl(\{\aa\}, \{\xa\}, \{\piab\}\bigr)$ $(0\leq \alpha\leq\beta\leq\tau)$ such that
\begin{enumerate}
\item[(i)] for each $0\leq \gamma\leq \Omega$ there are distinct points $x_\gamma^{(1)},\ldots, x_\gamma^{(n)}$ in $X_\gamma$ such that $\pi^{-1}_{\alpha,\beta}(x_\alpha^{(j)})= x_\beta^{(j)}$ for each $1\leq j \leq n$, and
$0\leq \alpha\leq\beta\leq\Omega$, and
\item[(ii)] for each $0\leq\alpha<\Omega$, the pair $(A_{\alpha+1},X_{\alpha+1})$ is the $\sf_\alpha$-extension of $(A_\alpha, X_\alpha)$ with $\sf_\alpha= M_{x_\alpha^{(1)}} \cap \cdots \cap M_{x_\alpha^{(n)}}$.
\end{enumerate}
Set $\sf=M_{x_\Omega^{(1)}} \cap \cdots \cap M_{x_\Omega^{(n)}}$, and note that $\sf=\bigcup\limits_{0\leq \alpha<\Omega} \sf_\alpha$.  Consequently, every function in $\sf$ has a square root in $\sf$.
For notational convenience, set $\tA=A_\Omega$, set $x_j=x^{(j)}_\Omega$ ($j=1,\ldots, n$),  and set $P=\{x_1,\ldots,x_n\}$.

By Lemma~\ref{preserve-normal2}, $\tA$ is normal.  Applying Lemma~\ref{F-H} then yields that $\tA$ has bounded relative units at every point of $X\sm P$.  Therefore, by Lemma~\ref{bru-implies-gen-peak-point}, every point of $X\sm P$ is a generalized peak point, and hence, is a one-point Gleason part.  Consequently, to show that $P$ is a Gleason part, it is enough to show that each pair of points of $P$ lie in a common Gleason part.  
For that, note that given $h\in \tA$ with $\|h\|\leq 1$, the function $T_\Omega h$ is in $A$ with $\|T_\Omega h\|\leq 1$, and so given $1\leq j,k\leq n$,
\begin{align}
\bigl | h(x_j)-h(x_k) \bigr | &= \bigl | (T_\Omega h)(x_0^{(j)}) - (T_\Omega h)(x_0^{(k)}) \bigr | \nonumber\\
&\leq \| x_0^{(j)} - x_0^{(k)} \|;\nonumber
\end{align}
consequently $\| x_j -x_k \| \leq \| x_0^{(j)} - x_0^{(k)} \|$, and hence, $x_j$ and $x_k$ lie in a common Gleason part.

There are no nonzero point derivations on $\tA$ at points of $X\sm P$ because each of these points is a generalized peak point.  To show that there are no nonzero point derivations on $\tA$ at points of $P$, we consider a point $y\in P$ and a function $f\in M_y$ and will show that $f$ is in $M_y^2$.  There exists $g\in \tA$ such that $g(x_j)^2=f(x_j)$ for $j=1,\ldots, n$.  Note that $g$ is in $M_y$.  Clearly $f-g^2$ is in $\sf$, so there exists $h\in \sf$ such that $f-g^2=h^2$.  Then $h$ is in $M_y$, and hence so are $g\pm ih$.  Since $f=(g+ih)(g-ih)$ we obtain that $f$ is in $M_y^2$, as desired.

We have shown that $\tA$ satisfies the properties required of the algebra $B$ in Theorem~\ref{main-theorem-p2} with strong regularity replaced by the weaker condition of normality.  In particular, Theorem~\ref{main-theorem-p1} is proved.

To obtain a strongly regular uniform algebra, set $B=A_{x_1}\cap\cdots\cap A_{x_n}$, where we are using the notation introduced in the paragraph preceding Lemma~\ref{4.1}.  By Lemma~\ref{strong-reg}(ii), $B$ has bounded relative units at each point of $X\sm P$, and hence is strongly regular at each point of $X\sm P$ by Lemma~\ref{bru-implies-gen-peak-point}.  Moreover, by Lemma~\ref{strong-reg}(i), $B$ is also strongly regular at each point of $P$ as well.  Because $B$ is a subalgebra of $\tA$, points in a common Gleason part for $\tA$ must also lie in a common Gleason part for $B$.  In the present situation, that implies that the Gleason parts for $B$ coincide with the Gleason parts for $\tA$.  To show that $B$ has no nonzero point derivations, we argue just as we did for $\tA$ noting that the functions involved in the factorization $f=(g+ih)(g-ih)$ for $f$ in the ideal $M_y$ of $B$ can be chosen to lie in $B$.
\epf

\bpf[Proof of Theorem~\ref{main-theorem-p3}]
The proof is essentially the same as the one just given except that instead of using a root system with index set the ordinals less than or equal to $\Omega$, we simply form an infinite sequence of uniform algebras and take a direct limit once (as we did in the proof of Theorem~\ref{main-theorem-d1}), and rather than set $\sf_\alpha= M_{x_\alpha^{(1)}} \cap \cdots \cap M_{x_\alpha^{(n)}}$ we take $\sf_\alpha$ to be a \emph{countable dense subset} of 
$M_{x_\alpha^{(1)}} \cap \cdots \cap M_{x_\alpha^{(n)}}$.
\epf

\medskip
\centerline{\textsc{Acknowledgments}}
\smallskip
We have already expressed our gratitude to Dales and Feinstein for allowing us to include their work in the paper.
In addition, we thank them for valuable discussions and correspondence.

Some of the work presented here was carried out while the second author was a visitor at Indiana University.  He thanks the Department of Mathematics for its hospitality.

\end{document}